\documentclass{amsart}
\usepackage{graphicx}
\usepackage{amsfonts}

\usepackage{amssymb}
\usepackage{mathrsfs}
\usepackage{amsthm}
\usepackage{amsmath}
\newtheorem{theorem}{Theorem}[section]

\newtheorem{lemma}[theorem]{Lemma}

\newtheorem{remark}{Remark}

\newtheorem{claim}{Claim}

\newcommand{\hyperg}[4]{\: _2\! F_1 \left( #1,\, #2;\, #3;\, #4 \right)}

\begin{document}
\title[Lipschitz continuity of the solutions to the Dirichlet problems]{Lipschitz continuity of the solutions
to the Dirichlet problems for the invariant laplacians} %$\Delta_{\vartheta}$

\author{Congwen Liu}
\email{cwliu@ustc.edu.cn}
\address{CAS Wu Wen-Tsun Key Laboratory of Mathematics,
School of Mathematical Sciences,
University of Science and Technology of China\\
Hefei, Anhui 230026,
People's Republic of China.}
\thanks{
The first author was supported by the National Natural Science Foundation of China grant 11971453.}

\author{Heng Xu}
\email{xuheng86@mail.ustc.edu.cn}
\address{School of Mathematical Sciences,
University of Science and Technology of China\\
Hefei, Anhui 230026,
People's Republic of China.}


\begin{abstract}
This short note is motivated by an attempt to understand the distinction between the Laplace operator
and the hyperbolic Laplacian on the unit ball of $\mathbb{R}^n$, regarding the Lipschitz continuity of the solutions to the
corresponding Dirichlet problems.

We investigate the Dirichlet problem
\begin{equation*}%\label{eqn:Dirichlet}
\begin{cases} \Delta_{\vartheta} u = 0, & \text{ in }\, \mathbb{B}^n,\\
u=\phi, & \text{ on }\, \mathbb{S}^{n-1},
\end{cases}
\end{equation*}
where
\[
\Delta_{\vartheta} := (1-|x|^2) \bigg\{ \frac {1-|x|^2} {4} \Delta  + \vartheta \sum_{j=1}^n x_{j} \frac {\partial } {\partial x_j}
+ \vartheta \left( \frac {n}{2}-1- \vartheta \right) I\bigg\}.
\]
We show that the Lipschitz continuity of boundary data always
implies the Lipschitz continuity of the solutions if $\vartheta > 0$,
but does not when $\vartheta \leq 0$.
\end{abstract}


\maketitle

\section{Introduction}

Let $\mathbb{B}^n$ be the open unit ball in $\mathbb{R}^n$ ($n\geq 2$) and $\mathbb{S}^{n-1}$ the unit sphere.

It is known, even for $n=2$, that the Lipschitz continuity of the boundary data $\phi:\mathbb{S}^{n-1}\to \mathbb{R}^n$
does not imply the Lipschitz continuity of the solution of the Dirichlet boundary value problem
\begin{equation}\label{eqn:Dirichlet}
\begin{cases} \Delta u = 0, & \text{in }\, \mathbb{B}^n,\\
u=\phi, & \text{on }\, \mathbb{S}^{n-1}.
\end{cases}
\end{equation}
See \cite[Example 1]{AKM08} for a counterexample. In \cite{AKM08}, the authors also showed that
the Lipschitz continuity of $\phi$ implies the Lipschitz continuity of the solution
provided that the harmonic extension $P[\phi]$ of $\phi$ is a $K$-quasiregular mapping.
Kalaj \cite{Kal08} obtained a related result, but under additional assumption of $C^{1,\alpha}$ regularity of
$\phi$.

Recently, Chen et al. \cite{CHRW18} investigate solutions of the hyperbolic Poisson equation $\Delta_h u = \psi$,
where
\[
\Delta_{h} u(x) := ( 1 - | x | ^ { 2 })^2  \Delta u(x) + 2 (n-2) (1-|x|^2)  \sum_{j=1}^n x_{j} \frac { \partial } { \partial x _ { j } } u(x)
\]
is the Laplace-Beltrami operator on the unit ball $\mathbb{B}^n$.
Among the other things, the authors showed that, in contrast with the above,
the Lipschitz continuity of $\phi$ implies the Lipschitz continuity of the solution
of the Dirichlet boundary value problem
\begin{equation}\label{eqn:HypDir}
\begin{cases}
\Delta_h u=0, & \text{in }  \,  \mathbb{B}^n,\\
 u=\phi,        & \text{on } \, \mathbb{S}^{n-1},
\end{cases}
\end{equation}
with no other assumption than that $n\geq 3$. See \cite[Theorem 1.2]{CHRW18}.

This short note is motivated by an attempt to understand the distinction between the two boundary value problems.
To this end, we consider a family of differential operators
\[
\Delta_{\vartheta} u(x) := (1-|x|^2) \bigg\{ \frac {1-|x|^2} {4} \Delta u(x) + \vartheta \sum_{j=1}^n x_{j} \frac {\partial u} {\partial x_j} (x)
+ \vartheta \left( \frac {n}{2}-1- \vartheta \right) u(x)\bigg\}.
\]
%Note that

We call $\Delta_{\vartheta}$ the (M\"obius) invariant Laplacians, since
\[
\Delta_{\vartheta} \left\{\left(\mathrm{det}\, \psi^{\prime}(x)\right)^{\frac {n-2-2\theta}{2n}} f(\psi(x))\right\}
= \left(\mathrm{det} \psi^{\prime}(x)\right)^{\frac {n-2-2\theta}{2n}} \left(\Delta_{\vartheta} f\right)(\psi(x))
\]
for every $f\in C^2(\mathbb{B}^n)$ and for every M\"obius transformation $\psi$ (see \cite[Proposition 3.2]{LP09}).
These differential operators are closely related to polyharmonic functions. In particular, the differential operator
$L_{\vartheta} := \frac {4} {1-|x|^2} \Delta_{\vartheta}$ plays a crucial role in the modified Almansi representation
for polyharmonic functions in \cite{LPS21} as well as in the cellular decomposition theorem for polyharmonic functions in
\cite{BH14, LPS21}. We refer the reader to \cite{LP04, LP09, Olo14} for more information about these differential operators.

We investigate the Dirichlet problem
\begin{equation}\label{eqn:alphaDirichlet}
\begin{cases} \Delta_{\vartheta} u = 0, & \text{ in }\, \mathbb{B}^n,\\
u=\phi, & \text{ on }\, \mathbb{S}^{n-1},
\end{cases}
\end{equation}
which includes as special cases the Dirichlet problems \eqref{eqn:Dirichlet} and \eqref{eqn:HypDir}, in view of
\[
\Delta_{0} =\frac {(1 -|x|^2)^2}{4} \Delta \quad \text{and} \quad
\Delta_{\frac {n}{2}-1} = \frac {1}{4} \, \Delta_h.
\]
Unlike in \cite{AKM08, CHRW18}, for convenience, we work with functions instead of mappings.

\vskip4pt

Our main result is as follows.

\vskip4pt

\begin{theorem}\label{thm:main}
Suppose $n \geq 2$ and $\vartheta > -1/2$.
\begin{enumerate}
\item[(i)]
If $\vartheta \leq 0$, there exists a Lipschitz continuous function $\phi: \mathbb{S}^{n-1} \to \mathbb{R}$
such that the solution of \eqref{eqn:alphaDirichlet}
%\begin{equation}\label{eqn:alphaDirichlet}
%\begin{cases} \Delta_{\vartheta} u = 0, & \text{ in }\, \mathbb{B}^n,\\
%u=\phi, & \text{ on }\, \mathbb{S}^{n-1}
%\end{cases}
%\end{equation}
is not Lipschitz continuous.
\item[(ii)]
If $\vartheta > 0$, the Lipschitz continuity of boundary data always
implies the Lipschitz continuity of the solutions. % of \eqref{eqn:alphaDirichlet}.
\end{enumerate}
\end{theorem}

\begin{remark}
%$\vartheta=0$ is a cut-point.
When $\vartheta=n/2-1$, $\Delta_{\vartheta} = \frac {1}{4} \, \Delta_h$ and the assumption $\vartheta>0$ reads $n>2$.
Theorem \ref{thm:main} explains why the assumption $n\geq 3$ in \cite[Theorem 1.2]{CHRW18} is necessary.
\end{remark}

\section{Preliminaries}

%\subsection{Hypergeometric functions}

In this section, we gather some technical lemmas needed for Sections 3 and 4, and fix our
notation.

According to \cite{LP04}, the Poisson kernel associated to $\Delta_{\vartheta}$ is given by
\[
P_{\vartheta}(x,\zeta):= c_{n,\vartheta} \frac {(1-|x|^2)^{1+2\vartheta}} {|x-\zeta|^{n+2\vartheta}},
\qquad (x,\zeta)\in \mathbb{B}^n\times \mathbb{S}^{n-1},
\]
with $c_{n,\vartheta} := \frac {\Gamma\left(\frac{n}{2}+\vartheta \right) \Gamma(1+\vartheta)} {\Gamma\left(\frac {n}{2} \right)
\Gamma(1+2\vartheta)}$ and the Poisson integral of a function $\phi\in C(\mathbb{S}^{n-1})$ is given by
\begin{equation}\label{eqn:alphaPoissonintegral}
\mathcal{P}_{\vartheta}[\phi](x):= \int_{\mathbb{S}^{n-1}} P_{\vartheta}(x,\zeta) \phi(\zeta) d\sigma(\zeta),
\end{equation}
where $\sigma$ is the surface measure on $\mathbb{S}^{n-1}$ normalized so that $\sigma(\mathbb{S}^{n-1})=1$.

\vskip4pt

We record the following from \cite[Theorem 2.4]{LP04}.

\vskip4pt
\begin{lemma}
The Dirichlet problem \eqref{eqn:alphaDirichlet}
%\begin{equation}\labe
%\begin{cases} \Delta_{\vartheta} u = 0, & \text{ in }\, \mathbb{B}^n,\\
%u=\phi, & \text{ on }\, \mathbb{S}^{n-1}
%\end{cases}
%\end{equation}
has a solution for all $\phi\in C(\mathbb{S}^{n-1})$ if and only if $\vartheta>-1/2$.
Moreover, when $\vartheta>-1/2$, the solution is given by $u=\mathcal{P}_{\vartheta}[\phi]$.
%\begin{equation}\label{eqn:alphaPoissonintegral}
%u(x)=\mathcal{P}_{\vartheta}[\phi](x):= c_{n,\vartheta} \int_{\mathbb{S}^{n-1}} \frac {(1-|x|^2)^{1+2\vartheta}} {|x-\zeta|^{n+2\vartheta}} \phi(\zeta) d\sigma(\zeta)
%\end{equation}
%where $c_{n,\vartheta} := \frac {\Gamma\left(\frac{n}{2}+\vartheta \right) \Gamma(1+\vartheta)} {\Gamma\left(\frac {n}{2} \right)
%\Gamma(1+2\vartheta)}$.
\end{lemma}

\vskip4pt

Straightforward calculation yields the following.
\vskip4pt

\begin{lemma}
For $k=1,\ldots, n$, we have
\begin{align}\label{eqn:diffPoissonknl}
\frac{\partial P_{\vartheta}}{\partial x_k} (x, \zeta) ~=~& -2(1+2\vartheta) c_{n,\vartheta} \frac{(1-|x|^2)^{2\vartheta}x_k|x-\zeta |^2} {|x-\zeta |^{n+2\vartheta+2}}\\
&\quad -~ (n+2\vartheta) c_{n,\vartheta} \frac{(1-|x|^2)^{1+2\vartheta}(x_k-\zeta_k)}{|x-\zeta |^{n+2\vartheta+2}}.\notag
\end{align}
\end{lemma}

A number of hypergeometric functions will appear throughout.
We use the classical notation %$\hyperg{a}{b}{c}{\lambda}$ to denote
\begin{equation*}\label{eq:hypergdefin}
\hyperg{a}{b}{c}{\lambda}:=\sum_{k=0}^{\infty}\frac{(a)_k(b)_k}{(c)_k}\frac{\lambda^k}{k!}
\end{equation*}
with $c\neq 0, -1,-2,\ldots$, where $(a)_k$ stands for the Pochhammer
symbol, which is defined as
\[
(a)_k := \begin{cases}
1,& \text{ if } k=0,\\
a(a+1)\ldots(a+k-1), &\text{ if } k\geq 1.
\end{cases}
\]
%This series gives an analytic function for $|\lambda|<1$,
%called the Gauss hypergeometric function associated to $(a,b,c)$.

We refer to \cite[Chapter 2]{AAR99} for the properties of these
functions. Here, we only record two formulas for later reference.
\begin{align}
%\hyperg{a}{b}{c}{1} ~=~& \frac {\Gamma(c) \Gamma(c-a-b)}
%{\Gamma(c-a) \Gamma(c-b)},\qquad \RePt(c-a-b)>0.
%\label{eqn:gauss}\\
\hyperg{a}{b}{c}{\lambda} ~=~& (1-\lambda)^{c-a-b} \hyperg{c-a}{c-b}{c}{\lambda}. \label{eqn:euler}\\
\frac {d}{d\lambda} \hyperg{a}{b}{c}{\lambda} ~=~& \frac {a b}{c} \hyperg {a+1}{b+1}{c+1}{\lambda}.
\label{eqn:diffhyperg}
%\hyperg{a}{b}{c}{\lambda} ~=~& \frac{\Gamma(c)}{\Gamma(b)\Gamma(c-b)}\int_0^1
%t^{b-1}(1-t)^{c-b-1} (1-t\lambda)^{-a} dt,\notag\\
%&\hspace{64pt} \RePt c>\RePt b>0;\;  |\arg(1-\lambda)|<\pi. \label{eqn:euler2}
\end{align}

\begin{lemma}
We have
\begin{equation}\label{eqn:intknl}
\int_{\mathbb{S}^{n-1}} P_\vartheta(x, \zeta) ~d\sigma(\zeta)
~=~ c_{n,\vartheta}\, \hyperg{-\vartheta}{\frac{n}{2}-1-\vartheta}{\frac n2}{|x|^2}
\end{equation}
and for $k=1,\ldots,n$,
\begin{equation}\label{eqn:intpartialknl}
\int_{\mathbb{S}^{n-1}} \frac {\partial P_{\vartheta}}{\partial x_k} (x, \zeta) ~d\sigma(\zeta)
~=~ C(n,\vartheta)\, \hyperg{1-\vartheta}{\frac{n}{2}-\vartheta}{\frac n2+1}{|x|^2} x_k,
\end{equation}
with
\[
C(n,\vartheta):= \frac{-2\vartheta(n -2-2\vartheta)}{n}\, c_{n,\vartheta}.
\]
\end{lemma}

\begin{proof}
The first identity follows from the formula (\cite[Lemma 2.1]{LP04})
%for $\lambda\in \mathbb{R}$,
\begin{equation}\label{eqn:keylem0}
\int_{\mathbb{S}^{n-1}} \frac{d\sigma(\zeta)}{|x-\zeta|^{2\lambda}} ~=~ \hyperg{\lambda}{\lambda-\frac{n}{2}+1}{\frac n2}{|x|^2},
\quad x\in \mathbb{B}^n,
\end{equation}
and \eqref{eqn:euler}. The identity \eqref{eqn:intpartialknl} follows by
differentiating both sides of \eqref{eqn:intknl} and applying \eqref{eqn:diffhyperg}.
\end{proof}

%This, together with \eqref{eqn:euler}, yields the following.
%
\begin{lemma}\label{cor:bddintegral}
When $\vartheta > 0$, the functions
\begin{equation*}
x ~\longmapsto~ \int_{\mathbb{S}^{n-1}} \frac {\partial P_{\vartheta}}{\partial x_k} (x, \zeta) ~d\sigma(\zeta),
\quad k=1,\ldots,n,
\end{equation*}
are all bounded on $\mathbb{B}^n$. %Moreover,
%\begin{equation}\label{eqn:intpartialknl}
%\int_{\mathbb{S}^{n-1}} \frac {\partial P_{\vartheta}}{\partial x_k} (x, \zeta) ~d\sigma(\zeta)
%~=~ C(n,\vartheta)\, \hyperg{1-\vartheta}{\frac{n}{2}-\vartheta}{\frac n2+1}{|x|^2} x_k,
%\end{equation}
\end{lemma}

\begin{proof}
It follows from \eqref{eqn:intpartialknl} that
\begin{align*}
\left|\int_{\mathbb{S}^{n-1}} \frac {\partial P_{\vartheta}}{\partial x_k} (x, \zeta) ~d\sigma(\zeta)\right|
~\lesssim~& \hyperg{1-\vartheta}{\frac{n}{2}-\vartheta}{\frac n2+1}{|x|^2} \\
~=~& \sum_{j=0}^{\infty} \frac {(1-\vartheta)_j (\frac{n}{2}-\vartheta)_j} {j!\, (\frac n2+1)_j} |x|^{2j}.
\end{align*}
Here and below, we use the notation $f_1\lesssim f_2$ if there exists a uniform constant $C>0$,
such $f_1\leq C f_2$. Also, $f_1\approx f_2$ means that $f_1\lesssim f_2$ and $f_2\lesssim f_1$.
Note that the coefficients in the last series are of order $j^{-1-2\vartheta}$, as $j\to \infty$.
This proves the assertion.
\end{proof}

%The purpose of this section is to prove the Lipschitz continuity of $\Phi = P _ { \vartheta } [ \phi ]$ when $\vartheta > 0$. Let we denote $\zeta _ { 0 } = e _ { 1 } = ( 1,0 , \cdots , 0 ) \in \mathbb{S} ^ { n - 1 }$, then we have

\vskip4pt

The following lemma may be well known,  we include a short proof for the convenience of the reader.

\vskip4pt

\begin{lemma}\label{lem:keylem}
Suppose that $p> q \geq 0$. We have
\begin{equation}\label{eqn:keylem}
\int_{\mathbb{S}^{n-1}} \frac{|\zeta-e_1|^{q}}{|\zeta-re_1|^{n-1+p}}\, d\sigma(\zeta) ~\lesssim~ \frac{1}{(1-r)^{p-q}},
\qquad r\in [0,1).
\end{equation}
\end{lemma}

\begin{proof}
The special case $q=0$ of the lemma is well-known, see for instance \cite[Lemma 2.9]{JP99}.

For $q>0$, noting that
\[
|\zeta-e_1|^{q} \leq \left(|\zeta-r e_1|+|re_1-e_1|\right)^q \leq 2^q \big[|\zeta-r e_1|^q+(1-r)^q \big],
\]
we have
\begin{align*}
\int_{\mathbb{S}^{n-1}} \frac{|\zeta-e_1|^{q}}{|\zeta-re_1|^{n-1+p}}\, d\sigma(\zeta)
~\lesssim~& \int_{\mathbb{S}^{n-1}} \frac{d\sigma(\zeta)}{|\zeta-re_1|^{n-1+p-q}}\\
&\quad ~+~ (1-r)^q \int_{\mathbb{S}^{n-1}} \frac{d\sigma(\zeta)}{|\zeta-re_1|^{n-1+p}},
\end{align*}
and \eqref{eqn:keylem} follows from the special case $q=0$.
\end{proof}

\section{Proof of Theorem \ref{thm:main}, Part (i)}

In the case $\vartheta=0$, as is mentioned in the introduction, \cite[Example 1]{AKM08} is such an example.

Now we assume that $\vartheta<0$. Take $\phi\equiv 1$. By \eqref{eqn:intpartialknl} and \eqref{eqn:diffhyperg},
\begin{align*}
\sum_{k=1}^{n} x_k \frac {\partial}{\partial x_k} \mathcal{P}_{\vartheta}[\phi] (x)
~=~&  C(n,\vartheta)\, \hyperg{1-\vartheta}{\frac{n}{2}-\vartheta}{\frac {n}{2}+1}{|x|^2}\, |x|^2\\
=~& C(n,\vartheta)\, \hyperg{\frac{n}{2}+\vartheta}{1+\vartheta}{\frac {n}{2}+1}{|x|^2}\,
|x|^2 \, (1-|x|^2)^{2\vartheta}.%\approx~& (1-|x|^2)^{2\vartheta}.
\end{align*}
Since $\vartheta<0$, by the same argument as in the proof of Lemma \ref{cor:bddintegral}, we see that the hypergeometric function in the last line is bounded both from above and below, hence
%Note that
\begin{equation}\label{eqn:bdybhvr}
\sum_{k=1}^{n} x_k \frac {\partial}{\partial x_k} \mathcal{P}_{\vartheta}[\phi] (x)  ~\approx~ (1-|x|^2)^{2\vartheta}
\end{equation}
near the boundary $\mathbb{S}^{n-1}$.
Note that
\[
\big|\nabla \left(\mathcal{P}_{\vartheta}[\phi]\right) (x) \big| ~\geq~ \left| \sum_{k=1}^{n} x_k \frac {\partial}{\partial x_k} \mathcal{P}_{\vartheta}[\phi] (x) \right|.
\]
So \eqref{eqn:bdybhvr} implies that $\nabla \left(\mathcal{P}_{\vartheta}[\phi]\right)$ is unbounded in $\mathbb{B}^n$.
Consequently,  $\mathcal{P}_{\vartheta}[\phi]$ is not Lipschitz continuous in $\mathbb{B}^n$.

\vskip4pt

\begin{remark}
When $n=3$, there is a more explicit example:
\begin{align*}
\mathcal{P}_{\vartheta}[1](x) ~=~& \frac {\Gamma\left(\frac {3}{2}+\vartheta\right) \Gamma(1+\vartheta)}{\Gamma\left(\frac {3}{2}\right) \Gamma(1+2\vartheta)}
\hyperg {-\vartheta}{\frac {1}{2}-\vartheta}{\frac {3}{2}} {|x|^2} \\
=~& 2^{-1-2\vartheta} \frac {(1+|x|)^{1+2\vartheta}-(1-|x|)^{1+2\vartheta}}{|x|},
\end{align*}
which is clearly not Lipschitz continuous when $\vartheta<0$. Here, in the last equality we have used the formula (see \cite[p. 39]{MOS66})
\[
(1+z)^{1-2a}-(1-z)^{1-2a} = 2z(1-2a) \hyperg {a}{a+\frac {1}{2}} {\frac {3}{2}} {z^2}.
\]
\end{remark}

\section{Proof of Theorem \ref{thm:main}, Part (ii)}

Let $\phi$ be a Lipschitz continuous function on $\mathbb{S}^{n-1}$, i.e.,
$|\phi(\zeta)-\phi(\eta)| \leq L |\zeta-\eta|$ for some constant $L$ and all $\zeta , \eta \in \mathbb{S}^{n-1}$.

We shall show that
\begin{equation}\label{eqn:gradient}
\sup_{x\in \mathbb{B}^n} |\nabla \left(\mathcal{P}_{\vartheta}[\phi]\right) (x)| < \infty,
\end{equation}
which would clearly imply the Lipschitz continuity of $\mathcal{P}_{\vartheta}[\phi]$.

To do this, we note that
\begin{equation}\label{eqn:compobyorth}
\left|\nabla \left(\mathcal{P}_{\vartheta}[\phi]\right)\right| =
\left| \left(\nabla \left(\mathcal{P}_{\vartheta}[\phi\circ T]\right)\right)\circ T^{-1}\right|
\end{equation}
holds for every $\phi\in C(\mathbb{S}^{n-1})$ and every orthogonal transformation $T$.
Also, given any orthogonal transformation $T$, the function $\phi\circ T$ is also
Lipschitz continuous on $\mathbb{S}^{n-1}$, with the same Lipschitz constant.
%
%So, to prove \eqref{eqn:gradient}, it suffices to show
%\begin{equation}\label{eqn:gradient2}
%\sup_{r>0} \left|\nabla \left(\mathcal{P}_{\vartheta}[\phi]\right) (re_1)\right| < \infty,
%\end{equation}
For any $x\in \mathbb{B}^n$, we can choose an orthogonal transformation $T$ such that $T^{-1}x=re_1$ with $r=|x|$,
so that
\[
\left|\nabla \left(\mathcal{P}_{\vartheta}[\phi]\right)(x)\right| =
\left| \left(\nabla \left(\mathcal{P}_{\vartheta}[\phi\circ T]\right)\right)(re_1)\right|.
\]

We are reduced to prove the following.

\vskip8pt

\begin{claim}
$\frac{\partial\mathcal{P}_{\vartheta}[\phi]}{\partial x_k}(re_1)$, $k=1,\ldots,n$, are bounded
functions of $r$ on $[0,1)$.
\end{claim}

\vskip8pt

%where $c_{n,\vartheta}$ is defined in Theorem \ref{thm1.1}.

We consider two different cases.

\vskip8pt

\noindent \textit{Case 1}: $k\in \{2, \ldots,n\}$.

\vskip8pt

By taking $x = re_1$ in \eqref{eqn:intpartialknl}, we have
\begin{align*}
\int_{\mathbb{S}^{n-1}} \frac{\partial P_{\vartheta}}{\partial x_k}(re_1, \zeta) ~d\sigma(\zeta) ~=~ 0,
\end{align*}
and hence
\begin{equation}\label{eqn:trivialone}
\frac{\partial\mathcal{P}_{\vartheta}[\phi]}{\partial x_k}(re_1) = \int_{\mathbb{S}^{n-1}} \frac{\partial P_{\vartheta}}{\partial x_k}(re_1, \zeta) [\phi(\zeta)-\phi(e_1)] ~d\sigma(\zeta).
\end{equation}

Now, taking $x = re_1$ in \eqref{eqn:diffPoissonknl}, we get
\begin{equation*}
\frac{\partial P_{\vartheta}}{\partial x_k} (re_1 , \zeta) = (n+2\vartheta) c_{n,\vartheta}\, \frac{(1-r^2)^{1+2\vartheta}\zeta_k}{|re_1-\zeta |^{n+2\vartheta+2}}.
\end{equation*}
Substituting this into \eqref{eqn:trivialone}, we obtain
\begin{align*}
\left |\frac{\partial\mathcal{P}_{\vartheta}[\phi]}{\partial x_k} (re_1) \right|
~\lesssim~&  (1-r^2)^{1+2\vartheta} \int_{\mathbb{S}^{n-1}}\frac{|\zeta_k|\, |\phi(\zeta)-\phi(e_1)|}{|re_1-\zeta |^{n+2\vartheta+2}}d\sigma(\zeta)\\
~\lesssim~& (1-r^2)^{1+2\vartheta} \int_{\mathbb{S}^{n-1}} \frac{|\zeta -e_1|^2d\sigma(\zeta )}{|re_1-\zeta |^{n+2\vartheta+2}},
\end{align*}
where we used the condition $|\phi(\zeta)-\phi(e_1)|\leq L |\zeta-e_1|$ and the easy observation that $|\zeta_k| \leq |\zeta -e_1|$ in the last inequality.
This, together with Lemma \ref{lem:keylem}, shows that $\frac{\partial\mathcal{P}_{\vartheta}[\phi]}{\partial x_k} (re_1)$ is a bounded
function of $r$ on $[0,1)$.
%and all $k\in \{2,\ldots,n\}$.%, we have, so we can get

\vskip8pt

\noindent \textit{Case 2}: $k=1$.

\vskip8pt

%\textbf{Case 2.} In this case we consider that $k=1$.
%Differentiate with the $1$st variable $x_1$ we have
It follows from \eqref{eqn:diffPoissonknl} that
\begin{align*}
\left|\frac{\partial P_{\vartheta}}{\partial x_1}(re_1,\zeta) \right|
~\leq~& 2(1+2\vartheta) c_{n,\vartheta} \frac {(1-r^2)^{2\vartheta}} {|re_1-\zeta |^{n+2\vartheta}}
~+~ (n+2\vartheta) c_{n,\vartheta} \frac {(1-r^2)^{1+2\vartheta}|r-\zeta_1|}{|re_1-\zeta |^{n+2\vartheta+2}}\\
~\leq~& 2(1+2\vartheta) c_{n,\vartheta} \frac {(1-r^2)^{2\vartheta}} {|re_1-\zeta |^{n+2\vartheta}}
~+~ (n+2\vartheta) c_{n,\vartheta} \frac {(1-r^2)^{1+2\vartheta}|\zeta-e_1|}{|re_1-\zeta |^{n+2\vartheta+2}}\\
& \qquad \quad
~+~ (n+2\vartheta) c_{n,\vartheta} \frac {(1-r^2)^{1+2\vartheta}(1-r)}{|re_1-\zeta |^{n+2\vartheta+2}}.
\end{align*}
Combining with the condition $|\phi(\zeta)-\phi(e_1)|\leq L |\zeta-e_1|$, this yields
\begin{align*}
\bigg|\frac{\partial\mathcal{P}_{\vartheta}[\phi]}{\partial x_1} & (re_1)\bigg| \\
~\leq~&
\left|\int_{\mathbb{S}^{n-1}} \frac{\partial P_{\vartheta}}{\partial x_1}(re_1,\zeta )\phi(e_1)d\sigma(\zeta )\right|
~+~ \left|\int_{\mathbb{S}^{n-1}} \frac{\partial P_{\vartheta}}{\partial x_1}(re_1,\zeta )[\phi(\zeta )-\phi(e_1)]d\sigma(\zeta )\right| \\
\lesssim ~ &  \underbrace{|\phi(e_1)|\, \left|\int_{\mathbb{S}^{n-1}} \frac{\partial P_{\vartheta}}{\partial x_1}(re_1,\zeta )d\sigma(\zeta ) \right|}_{I_1} ~+~ \underbrace{(1-r^2)^{2\vartheta} \int_{\mathbb{S}^{n-1}} \frac{d\sigma(\zeta)}{|re_1-\zeta|^{n+2\vartheta}}}_{I_2}\\
&\qquad \quad ~+~ \underbrace{(1-r^2)^{1+2\vartheta} \int_{\mathbb{S}^{n-1}} \frac{|\zeta-e_1|}{|re_1-\zeta|^{n+2\vartheta+2}} d\sigma(\zeta)}_{I_3}\\
&\qquad \qquad \quad ~+~ \underbrace{ (1-r^2)^{2+2\vartheta}\int_{\mathbb{S}^{n-1}} \frac{d\sigma(\zeta)}{|re_1-\zeta|^{n+2\vartheta+2}}}_{I_4}.
%&\qquad \qquad +~ \underbrace{\left|\int_{\mathbb{S}^{n-1}} \frac{\partial P_{\vartheta}}{\partial x_1}(re_1,\zeta )d\sigma(\zeta ) \right||\phi(e_1)|. }_{I_4}
\end{align*}
Since $\vartheta>0$, by Lemma \ref{cor:bddintegral}, $I_1$ is bounded. Also, in view of Lemma \ref{lem:keylem} , $I_2,I_3,I_4$ are all bounded.
Consequently, $\frac{\partial\mathcal{P}_{\vartheta}[\phi]}{\partial x_1} (re_1)$ is a bounded
function of $r$ on $[0,1)$.

This concludes the proof of the claim and hence the proof of the theorem.

\subsubsection*{Acknowledgement}
The first author was supported by the National Natural Science Foundation of China grant 11971453.

\subsubsection*{Data availability}
We do not analyse or generate any datasets, because our work proceeds within a theoretical and mathematical approach. One can obtain the relevant materials from the references below.

\end{document}